\documentclass[11pt]{amsart}

\usepackage{amsmath,amssymb,amsthm,amsmath}
\usepackage{geometry}
\usepackage{hyperref}
\usepackage{amsmath,amssymb,latexsym}
\newtheorem{theorem}{Theorem}

\newtheorem{definition}{Definition}
\newtheorem{proposition}{Proposition}
\newtheorem{remark}{Remark}
\numberwithin{equation}{section}

\title{Notes on Orlicz Spaces and Interpolation}
\author{M. Vali Siadat \& Wilfredo Urbina-Romero}
\begin{document}

\maketitle

\begin{abstract}
Orlicz spaces are a natural generalization of classical Lebesgue spaces.
They are built upon convex growth functions and provide a flexible framework
for studying integrability conditions beyond power-type behavior.
In this survey paper, we present the basic definitions, main properties, and
fundamental results concerning Orlicz spaces, and also discuss the Marcinkiewicz interpolation result for Orlicz spaces.
\end{abstract}

\section{Introduction}

Classical Lebesgue spaces \(L^p(\Omega)\) play a central role in modern analysis.
However, many problems in analysis, probability, and partial differential equations
require more refined notions of integrability. Orlicz spaces are a class of function spaces that generalize the classical Lebesgue spaces $L^p$. They were introduced by Władysław Orlicz in the 1930s to study integrability conditions that cannot be adequately described using power functions alone. Thus,  Orlicz extended the \(L^p\) scale
by replacing the power function \(t^p\) with a general convex function. Today, Orlicz spaces play an important role in functional analysis, harmonic analysis, partial differential equations, probability theory, and the calculus of variations.
 
 The construction of Orlicz spaces is based on the concept of a Young function.
 \begin{definition}
 A function
$\Phi : [0,\infty) \to [0,\infty]$
is called a Young function if it satisfies the following properties:
\begin{enumerate}
\item[1.]	$\Phi(0) = 0,$
\item[2.]	$\Phi(t) > 0 $ for all $t > 0,$
\item[3.]	$\Phi$ is convex. \footnote{From the conditions [1.], [2.], and [3.] it can be proved the alternative condition on $\Phi$: $\frac{\Phi(t)}{t}$ is increasing; which is a weaker condition on $\Phi$.}
\item[4.]	$\lim_{t \to \infty} \Phi(t) = \infty.	$
	\end{enumerate}
 \end{definition}
 
  \begin{definition}
 Let $(\Omega, \mu)$  be a measure space and let $\Phi$ be a Young function. The Orlicz space $L^\Phi(\Omega)$ consists of all measurable functions $f : \Omega \to \mathbb{R}$ such that
 \begin{equation}
\int_\Omega \Phi\left(\frac{|f(x)|}{\lambda}\right)\, d\mu(x) < \infty
\quad \text{for some } \lambda > 0.
\end{equation}
 \end{definition}
 
 This condition is called the \emph{Orlicz integrability condition}.

\begin{definition}
The Luxemburg norm of a function \(f \in L^\Phi(\Omega)\) is defined by
 \begin{equation}
\|f\|_{L^\Phi}
=
\inf \left\{
\lambda > 0 :
\int_\Omega \Phi\left(\frac{|f(x)|}{\lambda}\right)\, d\mu(x) \le 1
\right\}.
 \end{equation}
\end{definition}

Observe that  if $\Phi(t)=t^p$ for $t \geqslant 0$ is a Young function for any $ 1<p < \infty$ on $(\mathbb{R}^d, m)$ where $m$ is the Lebesgue measure, then it is easy to prove that 
$L ^\Phi$ coincides with $L^p(\mathbb{R}^d)$, the $L^p$ Lebesgue spaces, and the Luxemburg norm is equivalent to the usual \(L^p\)-norm. Thus, Orlicz spaces can be viewed as a natural extension of the classical Lebesgue scale, and additionally,
if $\Phi(t)=t \log ^{+} t \quad(t>0)$,\footnote{ $\log ^{+} b=\left\{\begin{array}{cc}\log b & \text { if } b \geqslant 1 \\ 0 & 0 \leq b<1\end{array}\right.$}  is also a Young function on $(\mathbb{R}^d, m),$ then
$L^\Phi$ coincides with  $L \log^{+} L,$  the Zygmund class.

\begin{theorem}
Let \((\Omega,\mu)\) be a measure space and let \(\Phi\) be a Young function.
Then the Orlicz space \(L^\Phi(\Omega)\), equipped with the Luxemburg norm,
is a Banach space.
\end{theorem}

\begin{proof}
It is well known that the Luxemburg functional satisfies:
\begin{enumerate}
  \item \(\|f\|_{L^\Phi} = 0\) if and only if \(f = 0\) almost everywhere,
  \item \(\|\alpha f\|_{L^\Phi} = |\alpha| \|f\|_{L^\Phi}\) for all scalars \(\alpha\),
  \item \(\|f+g\|_{L^\Phi} \le \|f\|_{L^\Phi} + \|g\|_{L^\Phi}\).
\end{enumerate}
The triangle inequality follows from the convexity of \(\Phi\).
Hence \(\|\cdot\|_{L^\Phi}\) is a norm on \(L^\Phi(\Omega)\).\\

Let $ \{f_n\}$ be a Cauchy sequence in \(L^\Phi(\Omega)\), then  for every \(\varepsilon > 0\), there exists \(N\) such that
\[
\|f_n - f_m\|_{L^\Phi} < \varepsilon
\quad \text{for all } n,m \ge N.
\]
By definition of the Luxemburg norm, this implies that
\[
\int_\Omega \Phi\left(\frac{|f_n(x) - f_m(x)|}{\varepsilon}\right)\, d\mu(x) \le 1
\quad \text{for all } n,m \ge N.
\]
Using standard arguments based on the convexity of \(\Phi\) and the above uniform
integrability estimate,  a subsequence \(\{f_{n_k}\}\)  can be extracted
which converges almost everywhere to some measurable function \(f\).

 The limit function $f$ belongs to \(L^\Phi(\Omega)\), since fixing $k$,
by Fatou's lemma and convexity of \(\Phi\), we have
\[
\int_\Omega \Phi\left(
\frac{|f(x) - f_{n_k}(x)|}{\varepsilon}
\right)\, d\mu(x)
\le
\liminf_{m \to \infty}
\int_\Omega \Phi\left(
\frac{|f_{n_m}(x) - f_{n_k}(x)|}{\varepsilon}
\right)\, d\mu(x)
\le 1.
\]
Hence,
\[
\|f - f_{n_k}\|_{L^\Phi} \le \varepsilon.
\]
This shows that \(f \in L^\Phi(\Omega)\).

Since the original sequence $\{f_n\}$ is Cauchy and a subsequence converges
in norm to \(f\), it follows that the entire sequence converges to \(f\) in the Luxemburg norm:
\[
\lim_{n \to \infty} \|f_n - f\|_{L^\Phi} = 0.
\]
Therefore, since every Cauchy sequence in \(L^\Phi(\Omega)\) converges to a limit in
\(L^\Phi(\Omega)\), it follows that \(L^\Phi(\Omega)\) is complete, and hence
a Banach space.
\end{proof}

\begin{definition}
Given a Young function \(\Phi\), its complementary function \(\Psi\)
is defined by
 \begin{equation}
\Psi(s) = \sup_{t \ge 0} \{ st - \Phi(t) \}.
 \end{equation}
\end{definition}
Observe that this implies Young's inequality immediately
$$ st  \leq  \Phi(t) +  \Psi(s), \quad s,t \geq 0$$
and $\Psi$ is also a Young function, 
\begin{theorem}
Let \(\Phi : [0,\infty) \to [0,\infty]$ be a Young function and let $\Psi : [0,\infty) \to [0,\infty] $  be its
complementary function i.e.,
\[
\Psi(s) = \sup_{t \ge 0} \{ st - \Phi(t) \}.
\]
Then \(\Psi\) is also a Young function.
\end{theorem}

\begin{proof}
We verify the defining properties of a Young function:
\begin{itemize}
\item  \(\Psi(0) = 0\). 
For \(s = 0\), we have
\[
\Psi(0) = \sup_{t \ge 0} \{ -\Phi(t) \}.
\]
Since \(\Phi(t) \ge 0\) for all \(t \ge 0\) and \(\Phi(0)=0\), it follows that
\[
\Psi(0) = 0.
\]
\item \(\Psi(s) \ge 0\) for all \(s \ge 0\), and \(\Psi(s) > 0\) for \(s>0\).

Taking \(t = 0\) yields
\[
\Psi(s) \ge s \cdot 0 - \Phi(0) = 0.
\]
If \(s > 0\), since \(\Phi(t)/t \to \infty\) as \(t \to \infty\) for a Young
function, there exists \(t>0\) such that \(st - \Phi(t) > 0\), hence
\(\Psi(s) > 0\).
\item Convexity of \(\Psi\).

For each fixed \(t \ge 0\), the function
\[
s \mapsto st - \Phi(t)
\]
is affine (hence convex).
The supremum of a family of convex functions is convex.
Therefore, \(\Psi\) is convex on \([0,\infty)\).

\item  \(\lim_{s \to \infty} \Psi(s) = \infty\).

Fix \(t>0\). Then for all \(s \ge 0\),
\[
\Psi(s) \ge st - \Phi(t).
\]
Letting \(s \to \infty\), the right-hand side tends to \(+\infty\).
Hence,
\[
\lim_{s \to \infty} \Psi(s) = \infty.
\]

\end{itemize}
 \end{proof}

\begin{theorem}[Generalized H\"older's Inequality]
Let \(\Phi\) and \(\Psi\) be complementary Young functions.
If \(f \in L^\Phi(\Omega)\) and \(g \in L^\Psi(\Omega)\), then
\[
\int_\Omega |f(x)g(x)|\, d\mu(x)
\le 2 \|f\|_{L^\Phi} \|g\|_{L^\Psi}.
\]
\end{theorem}

\begin{proof}
The proof relies on Young's inequality and the definition of the Luxemburg norm.

Since \(\Phi\) and \(\Psi\) are complementary Young functions, they satisfy
Young's inequality:
\[
st \le \Phi(t) + \Psi(s),
\quad \text{for all } s,t \ge 0.
\]

Assume first that
\[
\|f\|_{L^\Phi} \neq 0, \quad \|g\|_{L^\Psi} \neq 0.
\]
Define the normalized functions
\[
u(x) = \frac{|f(x)|}{\|f\|_{L^\Phi}},
\qquad
v(x) = \frac{|g(x)|}{\|g\|_{L^\Psi}}.
\]
By the definition of the Luxemburg norm, we have
\[
\int_\Omega \Phi(u(x))\, d\mu(x) \le 1,
\qquad
\int_\Omega \Psi(v(x))\, d\mu(x) \le 1.
\]

Applying Young's inequality point-wise with \(t = u(x)\) and \(s = v(x)\), we obtain
\[
u(x)v(x) \le \Phi(u(x)) + \Psi(v(x)).
\]
Multiplying both sides by
\(\|f\|_{L^\Phi} \|g\|_{L^\Psi}\) yields
\[
|f(x)g(x)|
\le
\|f\|_{L^\Phi} \|g\|_{L^\Psi}
\left[
\Phi\left(\frac{|f(x)|}{\|f\|_{L^\Phi}}\right)
+
\Psi\left(\frac{|g(x)|}{\|g\|_{L^\Psi}}\right)
\right].
\]

Then, integrating over \(\Omega\), we obtain
\begin{align*}
\int_\Omega |f(x)g(x)|\, d\mu(x)
&\le
\|f\|_{L^\Phi} \|g\|_{L^\Psi}
\left[
\int_\Omega \Phi\left(\frac{|f(x)|}{\|f\|_{L^\Phi}}\right)\, d\mu(x)
+
\int_\Omega \Psi\left(\frac{|g(x)|}{\|g\|_{L^\Psi}}\right)\, d\mu(x)
\right] \\
&\le
\|f\|_{L^\Phi} \|g\|_{L^\Psi} (1 + 1).
\end{align*}

Therefore,
\[
\int_\Omega |f(x)g(x)|\, d\mu(x)
\le 2 \, \|f\|_{L^\Phi} \, \|g\|_{L^\Psi}.
\]
The inequality is trivial if either norm is zero, so the proof is complete.
\end{proof}

We will see later, that the constant 2 is an artifact of the Luxemburg norm. Additionally, observe that when $\Phi(t) = t^p/p$ and $\Psi(s) = s^{p'}/p',$ the inequality reduces exactly to the classical Hölder's inequality:
$$ \int_\Omega |f(x)g(x)| \;dx \le \|f\|_p \|g\|_{p'}.$$

\begin{definition}
For a measurable function \(f\), the \emph{Orlicz norm} (also called the
\emph{Orlicz--dual norm}) is defined by
 \begin{equation}
\|f\|_{(L^\Phi)}
:=
\sup \left\{
\int_\Omega |f(x)g(x)|\, d\mu(x)
:
g \in L^\Psi(\Omega), \;
\rho_g= \int_\Omega \Psi(|g(x)|)\, d\mu(x) \le 1
\right\}.
 \end{equation}
\end{definition}

The space \(L^\Phi(\Omega)\) equipped with this norm is called an
\emph{Orlicz space with Orlicz norm}.

\begin{proposition}
The Orlicz norm satisfies:
\begin{enumerate}
  \item \(\|f\|_{(L^\Phi)} = 0\) if and only if \(f = 0\) almost everywhere,
  \item \(\|\alpha f\|_{(L^\Phi)} = |\alpha| \|f\|_{(L^\Phi)}\) for all scalars \(\alpha\),
  \item \(\|f+g\|_{(L^\Phi)} \le \|f\|_{(L^\Phi)} + \|g\|_{(L^\Phi)}\).
\end{enumerate}
Hence, \(\|\cdot\|_{(L^\Phi)}\) is a norm on \(L^\Phi(\Omega)\).
\end{proposition}

\begin{theorem}
For all \(f \in L^\Phi(\Omega)\),
 \begin{equation}
\|f\|_{L^\Phi}
\le
\|f\|_{(L^\Phi)}
\le
2 \|f\|_{L^\Phi}.
 \end{equation}
Consequently, the Luxemburg norm and the Orlicz norm are equivalent.
\end{theorem}

\begin{proof}
The right-hand inequality follows directly from the generalized H\"older
inequality:
\[
\int_\Omega |f(x) g(x)| dx
\le
2 \|f\|_{L^\Phi} \|g\|_{L^\Psi},
\]
and the normalization
\(\int_\Omega \Psi(|g|) d\mu\le 1\) implies \(\|g\|_{L^\Psi} \le 1\).

The left-hand inequality follows by choosing suitable extremal functions
\(g\) in the definition of the Orlicz norm and using Young's inequality.
\end{proof}

\begin{remark}.
\begin{enumerate}
  \item The Orlicz norm has a direct dual interpretation and is often preferred
  in duality arguments.
  \item The Luxemburg norm is technically simpler and more common in
  applications.
  \item Both norms generate the same topology on \(L^\Phi(\Omega)\), since they are equivalents norms.
\end{enumerate}
\end{remark}

\section{Dual Orlicz Spaces}

Let \((\Omega,\mu)\) be a $\sigma$-finite measure space and let $\Phi$ be a Young
function with complementary function $\Psi$.

\begin{theorem}
For every $g \in L^\Psi(\Omega)$, the linear functional
 \begin{equation}
\Lambda_g(f) := \int_\Omega f(x) g(x)\, d\mu(x),
\qquad f \in L^\Phi(\Omega),
 \end{equation}
is bounded on $L^\Phi(\Omega)$.
Moreover,
 \begin{equation}
\|\Lambda_g\|_{(L^\Phi)^*} \le 2 \|g\|_{L^\Psi}.
 \end{equation}
\end{theorem}

\begin{proof}
By the generalized H\"older inequality,
\[
|\Lambda_g(f)|
=
\left| \int_\Omega f(x) g(x)  dx\right|
\le
2 \|f\|_{L^\Phi} \|g\|_{L^\Psi},
\quad f \in L^\Phi(\Omega).
\]
Taking the supremum over $\|f\|_{L^\Phi} \le 1,$ yields
\(
\|\Lambda_g\| \le 2 \|g\|_{L^\Psi}.
\)
\end{proof}

Thus, $L^\Psi(\Omega)$ is continuously embedded into the dual space
$(L^\Phi(\Omega))^*$.
The converse inclusion requires additional structure.

\begin{definition}
A Young function \(\Phi\) satisfies the \(\Delta_2\)-condition if there exists
a constant \(C > 0\) such that
 \begin{equation}
\Phi(2t) \le C \Phi(t)
\quad \text{for all } t \ge 0.
 \end{equation}
\end{definition}

\begin{theorem}[Duality Theorem for Orlicz Spaces]
Assume that $\Phi$ satisfies the $\Delta_2$-condition.
Then every bounded linear functional
$\Lambda \in (L^\Phi(\Omega))^*$ can be represented in the form
 \begin{equation}
\Lambda(f) = \int_\Omega f(x) g(x)\, d\mu(x)
\quad \text{for some } g \in L^\Psi(\Omega).
 \end{equation}
Consequently,
\[
(L^\Phi(\Omega))^* \cong L^\Psi(\Omega)
\]
with equivalence of norms.
\end{theorem}

\begin{proof}[Sketch of proof]
The argument proceeds as follows:
\begin{enumerate}
  \item Using the boundedness of $\Lambda$, one shows that $\Lambda$ is
  continuous with respect to convergence in measure on sets of finite measure.
  \item By the $\sigma$-finiteness of $\mu$, the Riesz representation theorem
  yields a measurable function $g$ such that
  \[
  \Lambda(f) = \int_\Omega f g
  \quad \text{for all } f \in L^\Phi(\Omega) \cap L^\infty(\Omega).
  \]
  \item The $\Delta_2$-condition allows extension of this representation to
  all of $L^\Phi(\Omega)$ and guarantees that $g \in L^\Psi(\Omega)$.
\end{enumerate}
\end{proof}

If $\Phi$ does not satisfy the $\Delta_2$-condition, then
\[
L^\Psi(\Omega) \subsetneq (L^\Phi(\Omega))^*.
\]
In this case, the dual space contains additional \emph{singular functionals}
that cannot be represented by integration against a function in $L^\Psi$.\\

The  dual Orlicz space is complete, more precisely,

\begin{theorem}
Let $(\Omega,\mu)$ be a measure space and let $\Phi$ be a Young function.
Then,  the dual Orlicz space $(L^\Phi(\Omega))^*$, 
is a Banach space.
\end{theorem}

\begin{proof}
This is a consequence of a general principle in functional analysis:
the dual of any normed space is complete.

 Let $\{\Lambda_n\}$ be a Cauchy sequence in $(L^\Phi)^*$.
Then, by definition of the operator norm, for every $\varepsilon>0$ there exists
$N$ such that
\[
\|\Lambda_n - \Lambda_m\|_{(L^\Phi)^*} < \varepsilon
\quad \text{for all } n,m \ge N.
\]
Hence, for every $f \in L^\Phi(\Omega)$,
\[
|\Lambda_n(f) - \Lambda_m(f)|
\le
\|\Lambda_n - \Lambda_m\| \, \|f\|_{L^\Phi}
< \varepsilon \|f\|_{L^\Phi}.
\]
Now, fix $f \in L^\Phi(\Omega)$. The sequence $\{\Lambda_n(f)\}$ is Cauchy in $\mathbb{R}$
(or $\mathbb{C}$), hence converges.
Define
\[
\Lambda(f) := \lim_{n\to\infty} \Lambda_n(f).
\]
Since each $\Lambda_n$ is linear and limits preserve linearity,
the functional $\Lambda$ is linear on $L^\Phi(\Omega)$.

For $f \in L^\Phi(\Omega)$,
\[
|\Lambda(f)|
=
\lim_{n\to\infty} |\Lambda_n(f)|
\le
\limsup_{n\to\infty} \|\Lambda_n\| \, \|f\|_{L^\Phi}.
\]
Since $\{\Lambda_n\}$ is Cauchy in norm, it is bounded:
there exists $C>0$ such that $\|\Lambda_n\|\le C$ for all $n$.
Therefore,
\[
|\Lambda(f)| \le C \|f\|_{L^\Phi},
\]
and hence $\Lambda \in (L^\Phi(\Omega))^*$.\\

For $f$ with $\|f\|_{L^\Phi}\le 1$,
\[
|\Lambda_n(f) - \Lambda(f)|
=
\lim_{m\to\infty} |\Lambda_n(f) - \Lambda_m(f)|
\le
\limsup_{m\to\infty} \|\Lambda_n - \Lambda_m\|.
\]
Taking the supremum over $\|f\|\le 1$ gives
\[
\|\Lambda_n - \Lambda\|
\le
\limsup_{m\to\infty} \|\Lambda_n - \Lambda_m\|.
\]
Since $(\Lambda_n)$ is Cauchy, the right-hand side tends to zero as $n\to\infty$.
Thus,
\[
\lim_{n\to\infty} \|\Lambda_n - \Lambda\| = 0.
\]
Therefore, every Cauchy sequence in $(L^\Phi(\Omega))^*$ converges in norm to an element
of $(L^\Phi(\Omega))^*$.
Therefore, the dual Orlicz space is complete, i.e., a Banach space.
\end{proof}
\subsection{Reflexivity}

\begin{theorem}
The Orlicz space $L^\Phi(\Omega)$ is reflexive if and only if both
$\Phi$ and its complementary function $\Psi$ satisfy the $\Delta_2$-condition.
\end{theorem}

\begin{proof}[Idea of proof]
Reflexivity requires
\[
(L^\Phi)^* \cong L^\Psi
\quad \text{and} \quad
(L^\Psi)^* \cong L^\Phi,
\]
which holds precisely when both Young functions satisfy the $\Delta_2$-condition.
\end{proof}
\begin{remark}
The \(\Delta_2\)-condition guarantees important structural properties of
Orlicz spaces, such as separability and, in many cases, reflexivity.
It also plays a key role in identifying the dual space of \(L^\Phi(\Omega)\).
\end{remark}

\begin{theorem} Let $\phi: \mathbb{R} \to \mathbb{R}$ be a non-decreasing function and define
$$ \Phi(x) = \int_a^x \phi(u) du,$$
then, $\Phi$ is convex.
\end{theorem}

\begin{proof} 

Let $l(x) $ be the straight line joining the points $(a, \Phi(a))$ and $(b, \Phi(b))$ i.e., $$l(x) = \frac{\Phi(b) -\Phi(a)}{b-a} (x-a) + \Phi(a),$$
it suffices to show $\Phi(x) \leq l(x)$ for all $x\in (a,b)$  or $\frac{\Phi(x) -\Phi(a)}{x-a} \leq \frac{l(x) -l(a)}{x-a}  $ or equivalently
\begin{eqnarray*}
\frac{\int_a^x \phi(u) du }{x-a}  = \frac{\Phi(x) -\Phi(a)}{x-a} &\leq & \frac{\Phi(b) -\Phi(a)}{b-a}=  \frac{\int_a^b \phi(u) du}{b-a} \\
&=&  \frac{\int_a^b \phi(u) du+ \int_x^b \phi(u) du}{(x-a)+(b-x)} \\
&=& \frac{\int_x^b \phi(u) du }{b-x}.
\end{eqnarray*}
 Since $\phi$ is non-decreasing the extreme inequalities are true.\footnote{Remember if $\frac{a}{b} \leq \frac{c}{d}$ then $\frac{a}{b} \leq \frac{a+c}{b+d}\leq \frac{c}{d}$. }\\
 
 Consider 
 $$ \rho_{y/\rho_y} = \int_a^b \Psi\left(\frac{y}{\rho_y}\right) dt  \leq  \int_a^b \frac{\Psi(y)}{\rho_y}  dt =1 ,$$
 so, by the first inequality, since $ \rho_{y/\rho_y}  \leq 1,$
 $$ \left| \int_a^b x \frac{y}{\rho_y} dt \right| \leq \|x\|, \quad \text{implies} \quad  \left| \int_a^b x y  dt \right| \leq \|x\| \rho_y $$
 if $\rho_y >1$. Thus 
 $$ \left| \int_a^b x y  dt \right| \leq \|x\| \rho_y'$$
 where $\rho_y' =\text{Max}(1, \rho_y) $.\\
 
 We shall prove the theorem for $x$ bounded.\footnote{the boundedness of $x$ makes $\Phi(\frac{x}{\|x\|})$ and $\Psi(|y|)$ integrable.}
 
For the special case $y= \phi(\frac{x}{\|x\|})$, that of the equality in Young's inequality,
$$
\int_a^b \frac{x}{\|x\|} y dt = \int_a^b\Phi\left(\frac{x}{\|x\|}\right) dt +  \int_a^b\Psi\left(\frac{x}{\|x\|}\right) dt =\int_a^b\Phi\left(\frac{x}{\|x\|}\right) dt + \rho_y,
$$
then, 
$$ \rho_y' \geq \left| \int_a^b \frac{x}{\|x\|}  dt \right|  \geq \rho_y$$
assuming that $ \int_a^b\Phi\left(\frac{x}{\|x_n\|}\right) dt \neq 0.$ Hence $ \rho_y' \geq \rho_y$ and therefore  $\rho_y'=1.$

If $x$ is not bounded, we truncate it by $n$, i.e., let
$$ x_n (t) = \begin{cases}  x(t) &\text{if} \; |x(t)| \leq n \\
0 &\text{otherwise} 
\end{cases}
$$
For each $n$, $x_n$ is bounded so $ \int_a^b\Phi\left(\frac{x_n}{\|x\|}\right) dt \leq 1,$ for $n$ large enough such that $\|x_n\| \neq 0$, therefore,
$ \int_a^b\Phi\left(\frac{x_n}{\|x\|}\right) dt \leq 1,$  since $\|x_n\| \leq \|x\|$.
Thus, by Fatou's Theorem
\begin{eqnarray*}
1 \geq \liminf_n  \int_a^b\Phi\left(\frac{x_n}{\|x\|}\right) dt \geq \int_a^b \liminf_n  \Phi\left(\frac{x_n}{\|x\|}\right) dt = \int_a^b\Phi\left(\frac{x}{\|x\|}\right) dt, 
\end{eqnarray*}
since $\Phi$ is continuous.
\end{proof}
\begin{theorem}.
\begin{enumerate}
\item[i)]  If $u(x) = \int_a^b x(t) y(t) dt$ for every $x \in L_{\Phi}^*(a,b)$, then $y \in L_{\Psi}^*(a,b)$.
\item[ii)]  If for each $n$, $u_n(x) = \int_a^b x(t) y_n(t) dt$ is bounded for each $x \in L^*_\Phi(a,b),$ i.e. $|u_n(x) | \leq M_x$ for all $n,$ then $\|y_n\|_\Psi = o(1).$
\item[iii)] If for each $n$, $u_n(x) = \int_a^b x(t) y_n(t) dt$ is bounded for each $x\in  L^*_\Phi(a,b),$ then then there exists a constant $\theta>0$ such that 
 $$ \int_a^b \Psi(\theta |y_n|) \;dt =o(1).$$
 \item[iv)] If $\int_a^b x(t) y(t) dt$ exists for all $y \in L^*_\Psi$ with $\rho_{y,\Psi} = \int_a^b \Psi(|y|) dt \leq 1,$ then $\|x\|
 _\Phi \leq \infty.$

\end{enumerate}
\end{theorem}

\begin{proof}.

\begin{enumerate}
\item[i)]    Let $y_n(t)=\begin{cases} & y(t), \; \text{if}\; |y(t)| \leq n\\
 & 0 \quad \quad \text{elsewhere}
 \end{cases}
 $ 
 and let 
 $ u_n(t) = \int_a^b x(t)y_n(t) \;dt, \quad n =1, 2, \cdots $
 
For each $n$, $\rho_{y_n} = \int_a^b \Psi(|y_n|) \;dt < \infty,$ since $y_n$ is bounded by $n$.

Also, for each $n,$ 
$$ |u_n(x)| = \left| \int_a^b x(t)y_n (t) \;dt \right| \leq \|x\|_{\Phi} \rho'_{y_n},$$
since $x(t) \in L^*_{\Phi}(a,b)$, where, as before, $\rho'_{y_n} = \max(1, \rho_{y_n}).$
Therefore, $u_n(x)$ is a bounded linear functional in $L^*_\Phi(a,b).$

Additionally, by hypothesis, $u_n(x)$ converges for every $x \in L_{\Phi}^*(a,b),$ hence bounded, i.e., for each $x \in L_{\Phi}^*(a,b)$ $|u_n(x)| \leq M_x,$ for all $n,$ since
$$\lim_{n\to \infty} u_n(x) = \lim_{n\to \infty} \int_a^b x(t)y_n(t) \;dt = \int_a^b x(t)y(t) \;dt =u(x),$$
by Dominated Convergence Theorem. Note, $u_n(x) :  x \in L_{\Phi}^*(a,b), \to \mathbb{C}$ and 
$u_n(\alpha x + \beta y) = \alpha u_n(x) + \beta u_n(y),$ i.e., $u_n$ is linear. Therefore, by Banach-Steinhaus Theorem\footnote{ Let $B$ a Banach space (i.e. complete normed linear space), let $u_n(x): x \in B \to \mathbb{C})$ be a sequence of continuous linear functional on $B$ (i.e. for each $n$, $ |u_n(x)| \leq M_n \|x\|$ for all $x \in B$ and for each $x \in B$ , $|u_n(x)| \leq M_x$ for each $n$, then, there exists a constant $M>0$ such that $|u_n(x)| \leq M \|x\|$ for any $x \in B$ and for all $n$.} there exist a constant $M$ such that for all $n$
$$  |u_n(x)|  \leq M \|x\|_{\Phi}.$$

Now, let $x(t)$ be such that $\int_a^b \Phi(|x|) dt \leq 1,$ then $x \in L_{\Phi}$ and therefore $x \in L^*_{\Phi}$.
Also, by Young's inequality
\begin{eqnarray*}
\|x\|_{\Phi}  &=& \left| \sup_y  \int_a^b x(t)y(t) \;dt \right| \leq \sup_y  \int_a^b |x(t)|\,|y(t)| \,dt  \\
&\leq& \int_a^b \Phi(|x(t)|) \,dt  + \int_a^b \Psi(|y(t)|) \,dt  \leq 2,
\end{eqnarray*}
if $\rho_y \leq 1.$ Therefore $\left| \int_a^b y_n \; dt \right| \leq 2M$ for every $x(t)$ such that $\int_a^b \Phi(|x|) dt \leq 1.$  Then, $\|y_n\|_{\Psi}  \leq 2M,$ hence $|y\|_{\Psi}  \leq 2M$ and therefore $y\in L^*_{\Psi}.$

\item[ii)] By i), for each $n,$ $y_n \in L^*_\Psi(a,b)$, since $ x \in L^*_\Phi(a,b)$. Choose $\lambda_n >0$ such that $\lambda_n y_n \in L_\Psi(a,b).$ Therefore, for each $n,$  $ \lambda_n |u_n(x)| \leq \|x\|_\Phi \rho'_{\lambda y_n}.$ So, $u_n(x)$ is a continuous linear functional in $L^*_\Phi$ and then, by Banach-Steinhaus Theorem, $|u_n(x)| \leq M \|x\|_\Phi,$ for all $n$ and then 
$$ \left|  \int_a^b x(t)y_n(t) \;dt \right| \leq 2M, $$
since if $\rho_{x,\Phi} = \int_a^b \Phi(|x|) dt \leq 1,$ then $ \|x\|_\Phi \leq 2.$ Therefore, $\|y_n\|_{\Psi}  \leq 2M$  for all $n$ and hence $$\|y\|_{\Psi}  =o(1).$$
 
 \item[iii)] Suppose that $x \in L^*_\Phi(a,b)$, then  $\int_a^b \Phi\left(\frac{|x|}{\|x\|}\right) \, dt \leq 1,$ and therefore $\frac{x}{\|x\|} \in L_\Phi(a,b).$ So by assumption  $\int_a^b \frac{|x|}{\|x\|} y_n(t) \; dt$ is bounded for every $ x\in L^*_\Phi(a,b)$, i.e., $\left\{ \frac{u_n(x)}{\|x\|}\right\}$ is bounded for every $x \in L^*_\Phi,$ therefore $\|y_n\|_\Psi \leq N$ for all $n$ by ii). That implies $y_n \in L^*_\Psi(a,b).$ Hence $\int_a^b \Phi\left(\frac{|y_n|}{\|y_n\|_\Psi} \right)\, dt \leq 1$ and therefore  $\int_a^b \Phi\left(\frac{|y_n|}{N} \right)\, dt \leq 1$ , so $\int_a^b \Phi\left(\theta|y_n| \right)\, dt =o(1).$ 
 
 \item[iv)] Let  $y \in L^*_\Psi(a,b),$  then $\int_a^b \Phi\left(\frac{y_n}{\|y_n\|_\Psi} \right)\, dt \leq 1,$ and  hence $\frac{y_n}{\|y_n\|_\Psi} \in L^*_\Psi(a,b)$. So $\int_a^bx\frac{y_n}{\|y_n\|_\Psi} \, dt$ exists for any $y\in L^*_\Psi(a,b)$, and then, by i), $x \in L^*_\Phi(a,b),$ thus $\|x\|_\Phi < \infty.$ We conclude that 
 $$ L^*_\Psi = \{ x : |x\|_\Phi < \infty\} =  \{x: \int_a^b x(t) y(t) \, dt < \infty\}.$$
\end{enumerate}
\end{proof} 

Finally, we have the equivalence of $L_\Phi$ and $L^*_\Phi.$ 

We already know that if $x(t) \in L^*_\Phi$ if and only if there exists a constant $\theta >0$ such that $\theta x(t) \in L_\phi$. 

\underline{Claim}: there exists a constant $C$ such that 
$$ \Phi(2u) \leq C \Phi(u)$$
for $u$  large enough, and if $b-a <\infty,$ then $L_\Phi=L^*_\Phi.$ 

\begin{proof}
We already have $L_\Phi\subset L^*_\Phi,$  so we need to prove $L^*_\Phi\subset L_\Phi.$ \\

If $x \in L^*_\Phi,$ then there exists a constant $\theta >0$ such that $\theta x \in L_\Phi$. Now, by hypothesis, there exist $U$ such that $ \Phi(2u) \leq C \Phi(u)$ for $u > U.$
\begin{eqnarray*}
\int_\Omega  \Phi(2\theta |x|) d\mu  &=& \int_{\{t: \theta x >u\}}  \Phi(2\theta |x|) d\mu   +\int_{\{t: \theta x\leq u\}}   \Phi(2\theta |x|) d\mu  \\
&\leq & C \int_\Omega  \Phi(2\theta |x|) d\mu   + B < \infty.
\end{eqnarray*}
Therefore, $2\theta x \in L_\Phi$. Iterating this argument we get $2^k\theta x \in L_\Phi$, for $k=1, 2, \cdots$.
Hence, by choosing $k$ such that $2^k \theta >1, $ we conclude $x \in L_\Phi.$
\end{proof}

\section{Interpolation of Orlicz spaces}

Interpolation in Orlicz spaces is a natural extension of classical interpolation theory to the nonlinear setting generated by Young functions. It plays an important role in functional analysis, harmonic analysis, and operator theory.

Given a pair of Banach spaces $(X_0,X_1)$, interpolation theory aims to construct intermediate spaces $X_\theta$ such that operators bounded on $X_0$ and $X_1$ remain bounded on $X_\theta$.

In the Orlicz setting, one typically considers pairs:
\[
(L^{\Phi_0}, L^{\Phi_1})
\]
and seeks to characterize the resulting interpolation spaces.

\section{Real  and Complex Interpolation}

\begin{definition} The $K$-functional is defined by
\begin{equation}
K(t,f;L^{\Phi_0},L^{\Phi_1}) = \inf_{f=f_0+f_1} 
\left( |f_0|{L^{\Phi_0}} + t |f_1|{L^{\Phi_1}} \right).
\end{equation}
\end{definition}

In many situations, the interpolation space can again be identified with an Orlicz space:
\[
(L^{\Phi_0}, L^{\Phi_1}){\theta,q} = L^\Phi,
\]
where $\Phi$ is constructed from $\Phi_0$ and $\Phi_1$.

The classical example is the following: If
\[
\Phi_0(t)=t^{p_0}, \quad \Phi_1(t)=t^{p_1},
\]
then,
\[
\Phi(t)=t^p, \quad \frac{1}{p}=\frac{1-\theta}{p_0}+\frac{\theta}{p_1}, \; 0<\theta <1.
\]

The complex interpolation method yields, under the suitable assumption that the Orlicz spaces  satisfy the $\Delta_2$-condition, 
\begin{equation}
[L^{\Phi_0}, L^{\Phi_1}]\theta = L^\Phi,
\end{equation}
where $\Phi^{-1}$ satisfies the relation
\begin{equation}
\Phi^{-1}(t) = (\Phi_0^{-1}(t))^{1-\theta} (\Phi_1^{-1}(t))^\theta.
\end{equation}

For interpolation to behave well, the following additional conditions are typically required:
\begin{itemize}
\item convexity of $\Phi$,
\item regular behavior of $\Phi^{-1}$.
\end{itemize}

These ensure that the spaces are Banach and have good structural properties.\\

In this review, we will consider the extension of the Marcinkiewicz interpolation theorem to Orlicz spaces. In 1956,  A. Zygmund gave the proof of the general form of the Marcinkiewicz's theorem for the Lebesgue classes $L^p$. A. Torchinski exploited techniques used there to obtain Orlicz classes as intermediate spaces, as well.

\begin{definition}
 An operation $g=T f$ of a class of functions $b$ on $(M, \mu)$ into a class of functions $g$ on $(N, \nu)$ is called a sublinear operation if it satisfies the following properties:
 \begin{enumerate}
\item[(i)] If $f=f_0+f_1$ and $T f_i(i=0,1)$ are defined, then $T f$ is defined;
\item[(ii)] $\left|T\left(f_0+f_1\right)\right| \leqslant\left|T f_0\right|+\left|T f_1\right| \nu$-almost everywhere;
\item[(iii)] For any scalar $k$ we have $|T(k f)|=|k||T f| \nu$-almost everywhere.
\end{enumerate}
\end{definition}

We need one more definition. 

\begin{definition} Given the points $\left(\alpha_i, \beta_i\right), i=0,1$, $\alpha_0 \neq \alpha_1$, in the region $0 \leqslant \beta_i, \alpha_i<\infty$, we set
\begin{equation}\label{(2.1)}
\varepsilon=\frac{\beta_0-\beta_1}{\alpha_0-\alpha_1},
\end{equation}
and
\begin{equation}\label{(2.2)}
\gamma=\frac{\beta_0 / \alpha_0-\beta_1 / \alpha_1}{1 / \alpha_0-1 / \alpha_1}.
\end{equation}
Thus, the equation of the straight line passing through the points $\left(\alpha_i, \beta_i\right)$ is given by $y=\varepsilon \omega+\gamma$.
\end{definition}
 We can now state the extension of the Marcinkiewicz interpolation theorem to Orlicz spaces.
\begin{theorem} Let
$$
0 \leqslant \frac{1}{q_i}=\beta_i \leqslant \frac{1}{p_i}=\alpha_i<\infty, \quad i=0,1, \alpha_0 \neq \alpha_1, \beta_0 \neq \beta_1.
$$
Suppose that a sublinear operation $g=I f$ is simultaneously of weak types $\left(p_i, q_i\right), i=0,1$, with norms $M_0$ and $M_1$, respectively.

Assume that the Young's functions $\Phi, \Psi$ are given by
$$
\Phi(t)=\int_0^t \phi(s) d s, \quad \Psi(t)=\int_0^t \psi(s) d s
$$
where $\phi$ and $\psi$ are monotone. Further, assume that if $l=q_0 \vee q_1<\infty$ and $m=q_0 \wedge q_1$, then $\Psi(t) / t^l$ decreases, $\Psi(t) / t^m$ increases and
\begin{equation}\label{(2.4)}
 \int_0^t \Psi(s) / s^m \frac{d s}{s} \leqslant k_m \Psi(t) / t^m 
 \end{equation}
\begin{equation}\label{(2.5)} 
 \int_t^{\infty} \Psi(s) / s^l \frac{d s}{s} \leqslant k_l \Psi((t) / t^l.
\end{equation}

If $l=\infty$, we assume there is a $q>m$ such that $\Psi(t) / t^q$ decreases.
If $\varepsilon$ and $\gamma$ are given by (\ref{(2.1)}) and (\ref{(2.2)}) and if
\begin{equation}\label{(2.6)} 
\Psi^{-1}(t)=\Phi^{-1}\left(t^\varepsilon\right) t^\gamma,
\end{equation}
then $T$ is of type $(\Phi, \Psi)$ with norm $K=K\left(k_m, k_l, M_i, p_i, q_i\right)$.
\end{theorem}

\begin{proof}
 Without loss of generality, we may assume that $0<\dot{p}_0 <p_1 \leqslant \infty$. It is readily seen that $L_\Phi\subset L^{p_0}+L^{p_1}$. In fact, if for $f \in L_\Phi$ and $u>0$, we set $f_u=\operatorname{sgn} f \cdot(u \wedge|f|)$ and $f^u=f-f_u$, then our assumptions imply that $f_u$ is in $L^{p_1}$ and $f^u$ in $L^{p_0}$. Therefore, $T f$ is well defined for $f$ in $L_\Phi$.

To calculate $I=\int_N \Phi(|T f(y)| / 2 K) d v$ we use the well-known expression
$$
\begin{aligned}
I & =\int_0^{\infty} \phi(\lambda) m(T f, 2 K \lambda) d \lambda \\
& \leqslant \int_0^{\infty} \phi(\lambda) m\left(T f_u, K \lambda\right) d \lambda+\int_0^{\infty} \phi(\lambda) m\left(T f^u, K \lambda_0\right) d \lambda=I_1+I_2.
\end{aligned}
$$
The last inequality is an immediate consequence of the definition of the distribution function and the fact that $T$ is sub-linear.

Let us consider the case $0<q_0<q_1 \leqslant \infty$, thus $m=q_0, l=q_1$ and $\varepsilon>0$. In the above decomposition of $f=f^u+f_u=f_{u_0}+f_{u_1}$ we choose $u$ as the monotone function of $\lambda$ such that $u^{-1}(\lambda)=\Psi^{-1}\left(\Phi^{1 / 8}(\lambda)\right)$. Two cases arise: $q_1<\infty$ and $q_1=\infty$.
\begin{itemize}
\item $q_1<\infty$: Since $T$ is of weak type ( $p_i, q_i$ ) with norm $M_i$, we have
\begin{eqnarray}\label{(2.7)}
m\left(T f_{u_i}, K \lambda\right) && \leqslant\left(M_i\left|f_{u_i}\right| p_i\right)^{q_i} /(K \lambda)^{q_i} \\
\nonumber & =&M_i^{q_i} p_i^{q_i \mid p_i}\left(\int_0^{\infty} m\left(f_{u_i}, s\right) s^{p_i} \frac{d s}{s}\right)^{q_i \mid p_i}(K \lambda)^{q_i}, \quad i=0,1.
\end{eqnarray}
Since $m\left(f_{u_1}, s\right)=m\left(f_u, s\right)=m(f, s)$ for $u(\lambda) \leqslant s$ and zero otherwise, it follows that
$$
I_1^{p_1 / q_1} \leqslant p_1\left(M_1 / K\right)^{p_1}\left\{\int_0^{\infty} \frac{\psi(\lambda)}{\lambda^{q_1}}\left(\int \chi(s \mid u(\lambda)) m(f, s) s^{p_1} \frac{d s}{s}\right)^{q_1 / p_1} d \lambda\right\}^{p_1 / q_1},
$$
where $\chi$ is the characteristic function of the interval $(0,1)$. Therefore by Minkowski's integral inequality and the fact that $\psi(\lambda) \leqslant q_1 \Psi(\lambda) / \lambda,$ since $\Psi(\lambda) / \lambda^{q_1}$ decreases, we obtain
$$
\begin{aligned}
I_1^{p_1 / q_1} & \leqslant p_1\left(M_1 / K\right)^{p_1} \int_0^{\infty} m(f, s) s^{p_1}\left\{\int_0^{\infty} \chi(s / u(\lambda))^{\alpha_1 / p_1} \frac{b(\lambda)}{\lambda^{q_1}} d \lambda\right\}^{p_1 / q_1} \frac{d s}{s} \\
& \leqslant p_1\left(M_1 / K\right)^{p_1} q_1^{p_1 / q_1} \int_0^{\infty} m(f, s) s^{p_1}\left\{\int_{u^{-1}(s)}^{\infty} \Psi(\lambda) / \lambda^{q_1} \frac{d \lambda}{\lambda}\right\}^{p_1 / q_1} \frac{d s}{s}.
\end{aligned}
$$
To bound the innermost integral, by integral condition (\ref{(2.5)}) and relation (\ref{(2.6)}) it can be obtained that
$$
\begin{aligned}
\int_{u^{-1}(s)}^{\infty} \Psi(\lambda) / \lambda^{q_1} \frac{d \lambda}{\lambda} & \leqslant k_{q_1} \Psi\left(u^{-1}(s)\right) / u^{-1}(s)^{q_1} \\
& =k_{q_1} \Psi\left(\Psi^{-1}\left(\Phi(s)^{1 / s}\right)\right) / \Psi^{-1}\left(\Psi(s)^{1 / q}\right)^{q_1} \\
& \leqslant k_{q_1}\Phi(s)^{1 / s} /\left(s \Phi(s)^{s/ q}\right)^{q_1},
\end{aligned}
$$
Hence, it follows that
$$
I_1^{p_1 / q_1} \leqslant p_1\left(M_1 / K\right)^{p_1}\left(q_1^{l_k} q_1\right)^{p_1 / q_1} \int_0^{\infty} m(f, s) s^{p_1} \Phi(s)^{\frac{p_1}{s}\left(1 / q_1-\gamma\right)} \frac{1}{s^{p_1}} \frac{d s}{s} .
$$
But since
$$
\frac{p_1}{\varepsilon}\left(1 / q_1-\gamma\right)=1 \quad \text { and } \quad \frac{\Phi(s)}{s} \leqslant \frac{\phi(s)}{p_0},
$$
and $\Phi(s) / s^{p_0}$ increases, it follows that
$$
I_1 \leqslant\left(\frac{p_1}{p_0}\right)^{\alpha_1 / p_1}(M / K)^{q_1} q_1^{k_0} \alpha_1\left(\int_0^{\infty} m(f, s) \phi(s) \frac{d s}{s}\right)^{\alpha_1 / p_1} .
$$
The same idea can be used to bound $I_2$, in this argument. The letter $c$ will denote a constant that may not be the same in different occurrences. However, the final expression $c \leqslant\left(M_0 / K\right)^{\alpha_0} k_{\alpha_0} q_1$. Since $m\left(f_{u_0}, s\right)= m\left(f^u, s\right)=m(f, u+s)$, from (2.7) it follows that
$$
I_2^{p_0 / q_0} \leqslant c\left\{\int_0^{\infty} \frac{\psi(\lambda)}{\lambda_0}\left(\int_0^{\infty} \chi(u(\lambda) / s) m(f, s) s^{p_0} \frac{d s}{s}\right)_1^{q_0 / p_0} d \lambda\right\}_1^{p_0 / q_0},
$$
hence, by Minkowski's integral inequality,
$$
I_2^{p_0 / q_0} \leqslant c \int_0^{\infty} m(f, s) s^{p_0}\left\{\int_0^{u-1}(s) \Psi(\lambda) / \lambda^{q_0} \frac{d \lambda}{\lambda}\right\}^{p_0 / q_0} \frac{d s}{s} .
$$
By integral condition (\ref{(2.4)}) and relation (\ref{(2.6)}), it follows that
$$
I_2^{p_0 / q_0} \leqslant c \int_0^{\infty} m(f, s) s^{p_0} \Phi(s)^{\frac{p_0}{\varepsilon}\left(1 / q_0-\gamma\right)} / s^{p_0} \frac{d s}{s} .
$$
Thus, since $\frac{p_0}{\varepsilon}\left(1 / q_0-\gamma\right)=1$ and $\Phi(s) / s \leqslant \frac{a(s)}{p_0}$, we obtain
$$
I_2 \leqslant c\left\{\int_0^{\infty} m(f, s) \phi(s) d s\right\}^{q_0 / p_0} .
$$
Let now $K=2^{1 / \alpha_0} M_0 q_1^{1 / \alpha_0} k_{q_0}^{1 / q_0} \vee 2^{1 / q_1}\left(\frac{p_1}{p_0}\right)^{1 / p_1} M_1 q_1^{1 / q_1} k_{q_1}^{1 / q_1}$. It then readily follows that
$$
I \leqslant \frac{1}{2}\left\{\left(\int_0^{\infty} m(f, s) a(s) d s\right)^{p_0 / q_0}+\left(\int_0^{\infty} m(f, s) a(s) d s\right)^{p_1 / q_1}\right\}
$$
and consequently $\int_M \Phi(|f(x)|) d \mu \leqslant 1$ implies that $\int_N \Psi(|T f(y)| / 2 K) d \nu \leqslant 1$ and $T$ is of type $(\Phi, \Psi)$ with norm $\leqslant 2 K$.
\item If $q_1=\infty$, the notions of weak type and type coincide. First assume $p_1<\infty$. Since $\Psi(s) / s^{p_1}$ decreases,
$$
\begin{aligned}
\left\|T f_u\right\|_{\infty} & \leqslant M_1\left|f_u\right|_{p_1} \leqslant M_1\left\{p_1 \int_0^{u(\gamma)} \frac{s^{p_1}}{\Phi(s)} \Phi(s) m(f, s) \frac{d s}{s}\right\}^{1 / p_1} \\
& \leqslant\left(\frac{p_1}{p_0}\right)^{1 / p_1} M_1 u(\lambda) / \Phi(u(\lambda))^{1 / p_1}\left\{\int_0^{\infty} \phi(s) m(f, s) d s\right\}^{1 / p_1} \\
& \leqslant\left(\frac{p_1}{p_0}\right)^{1 / p_1} M_1 \Phi^{-1}\left(\Psi(\lambda)^s\right) / \Psi(\lambda)^{\omega p_1}
\end{aligned}
$$
But now, $\gamma+\varepsilon / p_1=0$. Whence by (\ref{(2.6)}) we have
$$
\left\|T f_u\right\|_{\infty} \leqslant\left(\frac{p_1}{p_0}\right)^{1 / p_1} M_1 \lambda,
$$
and $m\left(T f_u, K \lambda\right)=0$ provided $K \geqslant\left(p_1 / p_0\right)^{1 / p_1} M_1$. Also $I_2$ may be bounded as before, provided there is a $q, q_0<q<\infty$ such that $\Psi(s) / s^\alpha$ decreases, but this is the assumption. Thus, setting $K=\left(p_1 / p_0\right)^{1 / p_1} M_1 \vee M_0 q^{1 / q_0} k_{q_0}^{1 / q_0},$ it follows that
$$
I \leqslant\left\{\int_0^{\infty} \phi(s) m(f, s) d s\right\}^{p_0 / q_0} .
$$
Therefore,  $T$ is of type $(\Phi, \Psi)$ with norm $\leqslant 2 K$.
If $p_1=q_1=\infty$, then $\gamma=0$ and $1 / \varepsilon=q_0 / p_0$. Then, since
$$
\left\|T f_u\right\|_{\infty} \leqslant M_1\left\|f_u\right\|_{\infty}=M_1 u(\lambda)=M_1 \Phi^{-1}\left(\Phi(\lambda)^{p_0 / q_0}\right),
$$
from (\ref{(2.6)}) we obtain
$$
\left\|T f_u\right\|_{\infty} \leqslant M_1 \lambda,
$$
and $m\left(T f_u, K \lambda\right)=0$ whenever $K \geqslant M_1$. Set $K=M_1 \vee M_0 k_{q_0}^{1 / q_0} q^{1 / q_0}$. As before, it follows that $T$ is of type $(\Phi, \Psi)$ with norm $\leqslant 2 K$. We must now consider the case $0<q_1<q_0 \leqslant \infty$. Thus, if $\varepsilon>0$ and we set $u^{-1}(\lambda) =\Psi^{-1}\left(\Phi(\lambda)^{1 / s}\right)$, then $u(\lambda)$ will be monotone decreasing. Hence, if as above, we put $I \leqslant I_1^{\prime}+I_2^{\prime}$, it is not difficulty in computing a bound for $I$ by interchanging the roles of $I_1^{\prime}$ and $I_2^{\prime}$ with $I_2$ and $I_1$ respectively.
\end{itemize}

\end{proof}

\end{document}